\documentclass[11pt, twoside, a4paper]{article}
		\usepackage{amsmath,amssymb,amsthm,amscd,appendix,calrsfs,cite,color,epsfig,eucal,enumitem,dsfont,graphics,graphicx,latexsym,mathrsfs,verbatim,hyperref}
		\newtheorem{theorem}{Theorem}
		\newtheorem{definition}[theorem]{Definition}
		\newtheorem{corollary}[theorem]{Corollary}
		\newtheorem{lemma}[theorem]{Lemma}
		\newtheorem{proposition}[theorem]{Proposition}
		\newtheorem*{teorema}{Theorem}

                \newcommand{\keywords}[1]{\par\addvspace\baselineskip\noindent\textbf{Keywords:}\enspace\ignorespaces#1}

                \newcommand{\AMSclassification}[1]{\par\addvspace\baselineskip\noindent\textbf{Mathematical subject classification:}\enspace\ignorespaces#1}
	
		\DeclareMathOperator*{\argmax}{arg\,max}
		
		\DeclareMathOperator{\Dia}{diam}

       \title{A Ruelle-Perron-Frobenius theorem for \\ expanding circle maps with an indifferent  fixed point}
       
        \author{Eduardo Garibaldi\thanks{Supported by CNPq grant 304792/2017-9 and FAPESP grant 2019/10485-8.} \\ 
        \small{Department of Mathematics, University of Campinas, 13083-859 Campinas, Brazil} \\ \small{(email: garibaldi@ime.unicamp.br)} \\ ~ \\
         Irene Inoquio-Renteria\thanks{Supported by FONDECYT 11130341 and BCH-CONICYT postdoctoral fellowship 74170014.} \\ \small{ICFM, Universidad Austral de Chile, casilla 567 Valdivia, Chile} \\
               \small{(email: ireneinoquio@uach.cl)}}

\begin{document}
			\maketitle
			
			\begin{abstract}
			In this note, we establish an original result for the thermodynamic formalism in the context 
			of expanding circle transformations with an indifferent fixed point. For an observable whose 
			continuity modulus is linked to the dynamics near such a fixed point, by identifying an appropriate 
			linear space to evaluate the action of the transfer operator, we show that there is a strictly positive eigenfunction
			associated with the maximal eigenvalue given as the exponential of the topological pressure.
			Taking into account also the corresponding eigenmeasure, the invariant probability thus obtained is proved to be
			the unique Gibbs-equilibrium state of the system. 	
			
			
			\AMSclassification{37D25, 37E05,  37D35, 26A12, 26A15.}		
			\end{abstract}
	
\section{Introduction}

The pioneering works of Sinai, Ruelle and Bowen provided the impetus for the development of a fruitful area in ergodic theory of differentiable systems.
Thermodynamic formalism for uniformly hyperbolic systems and H\"older continuous potentials has today a well-established theoretical ground and contributions to extend it consider different contexts and approaches. 
In this work, we focus on expanding maps on the circle that have an indifferent fixed  point and we take into account potentials with a modulus of continuity whose only imposition is dictated by the behavior of the dynamics in a neighborhood of this fixed point.

For frameworks that are similar to ours, Ruelle-Perron-Frobenius type theorems and theorems about the existence and uniqueness of equilibrium states, central results of this article, are recorded in the literature. 
The main lines of research focus on potentials obeying properties related to fundamental entities of thermodynamic formalism, such as the topological pressure and the transfer operator, and/or belonging to certain specific classes of regularity.

Without any intention of being exhaustive, it is worth mentioning some works to illustrate advances in this current.
For piecewise monotone interval transformations, Hofbauer and Keller \cite{HK82} have studied equilibrium states associated with potentials of bounded variation whose oscillation is strictly smaller than the topological entropy. 
In the same dynamic context, Denker, Keller and Urba\'nski \cite{DKU90} proved that, on each topologically transitive component, there is at most one equilibrium state associated with a potential either of bounded variation or with bounded distortion under the transformation, which, besides one of these regularity conditions, has as its supremum a value strictly smaller than its topological pressure. 
Existence and uniqueness of equilibrium states were obtained by Liverani, Saussol and Vaienti \cite{LSV98} for a class of piecewise monotone transformations on a totally ordered, compact set and for observables named as contracting potentials. 
Among the properties requested to be a contracting potential \cite[Definition~3.4]{LSV98},  there is the demand that the supremum of one of its Birkhoff sums is strictly smaller than the logarithm of the infimum of the corresponding iterate of the transfer operator applied to the function identically equal one. 
For smooth interval maps, the condition on the potential introduced by Hofbauer and Keller was used by Bruin and Todd \cite{BT08} to propose a proof for existence and uniqueness of equilibrium states by means of an inducing scheme.  
In \cite{LR14}, for a sufficiently regular one-dimensional map satisfying a weak form of hyperbolicity, Li and Rivera-Letelier showed that, given a H\"older continuous potential, the supremum of one of its Birkhoff averages is strictly smaller than its topological pressure, a condition that guarantees in particular the existence and uniqueness of equilibrium states. 
Generalizing an optimal transportation method successfully applied for the thermodynamic formalism in the uniformly expanding context \cite{KLS15}, recently Kloeckner \cite{Klo20} was able to prove a Ruelle-Perron-Frobenius theorem and  to study equilibrium states for non-uniformly expanding maps and observables named as flat potentials. 
Flatness is a property that requires a uniform regularity control on all Birkhoff sums taken along pairing trajectories following a common transition kernel (see  \cite[Definition 2.13]{Klo20} for technical details).

One of the main contributions of our work is the identification of an easily verifiable property relating, in a neighborhood of the indifferent fixed point, the joint behavior of the dynamics and a pair of moduli of continuity  (see condition~\eqref{condicaosuficientecompatibilidade} below). Here one modulus describes the regularity of a potential and the other one  indicates the space on which the transfer operator's action should be considered in order to obtain relevant spectral information. 
This is a sufficient condition to ensure key results of the thermodynamic formalism from known methods and techniques, specially from potential theory or harmonic analysis. 
In practical terms, it is possible, for example, to fix a dynamics on the circle among the members of the analyzed family and without difficulty to determine possible classes of regularity of potentials for which a Ruelle-Perron-Frobenius theorem and the existence and uniqueness of  Gibbs-equilibrium states can be proved. 
In section~1.4 of \cite{GI20} there are illustrations of this simple compatibility procedure for a particular situation. 
We provide more general examples along this paper.

Roughly speaking, the condition of coherence between the dynamics around the indifferent fixed point and both classes of regularity guarantees that 
these moduli of continuity  are nicely related along backward orbits (see Definition~\ref{modulosepreimagens} and Propostion~\ref{compatibilidadesuficiente}).
This feature thus allows establishing a direct Ruelle-Peron-Frobenius theorem in 
a non-uniformly hyperbolic setting without inducing: we obtain a strictly positive eigenfunction 
of the transfer operator when looking at its action on the space of the functions with the suggested
regularity.
Existence and uniqueness of equilibrium state are discussed taking advantage of Rokhlin formula,
being useful to have the eigenequation to eliminate the possibility of the Dirac delta at the indifferent fixed point being an aspiring to equilibrium state.

In a preliminary version of this article, we did not emphasize the actual extent of our results.
In particular, it was pointed out that there is an overlap between our results and those 
obtained by  Kloeckner \cite{Klo20}. It should first be noted that, in cases where both works apply, 
the proofs are independent, and our strategy is dissociated from the approach in \cite{Klo20}, 
which is focused on determining a contraction rate of the dual of the transfer operator through couplings, a method introduced in \cite{Sta17}. Here the proposal of a comprehensive scenario by means of a compatibility between dynamics and moduli of continuity allows us to go further: 
the generality of the maps makes it very easy to present examples and to work with classes of regularity far beyond the usual H\"older modulus environment. 
Corollaries~\ref{corA} and~\ref{corB}, for instance, illustrate the fact that our results cover a range of examples and how simple can it be to ensure their application in wider situations.
Furthermore, from a mainly theoretical perspective, 
it is relevant to have a more accurate understanding of how the specific form of the indifferent fixed point influences thermodynamic formalism. 
In significant cases, one cannot deal with the H\"older regularity class, being necessary to consider different moduli to ensure a Ruelle-Perron-Frobenius theorem. 
The general formulation developed here seeks to contribute to a global understanding of these aspects, as clearly is the purpose of Kloeckner's work as well.

\subsection{Dynamics and Regularity Classes}

        Let $\mathbb T= \mathbb R/ \mathbb  Z= [0,1)$ denote the circle endowed with the standard metric
					$$ d(x,y) = \min\{|x-y|, \, |x-y \pm 1|\}, \quad x, y \in [0,1). $$
        We consider a family $ \mathscr F$ of continuous maps  $ T:\mathbb T\to \mathbb T$ which are non-uniformly 
				expanding with an indifferent fixed point. More precisely, we suppose that $T$ is of the form $T(x):=x(1+ V(x))\mod 1$, for all $ x \in [0, 1) $,
        where the continuous and increasing function $V:[0,+\infty)\to [0,+\infty)$ satisfies that $ V(1) $ is a positive integer and for some $\sigma\ge 0$
        \begin{equation*}
        \lim_{x\to 0} \frac{V(tx)}{V(x)} = t^{\sigma}, \textrm{ for all } t>0.
        \end{equation*}
        When $\sigma>0$, $V$ is called \emph{regularly varying with index $\sigma$}, and when $\sigma=0$, $V$ is called \emph{slowly varying}. 
	For the main properties of these families of functions, we refer the reader to \cite{Sen76}.
	We remark that any map $ T $ in $ \mathscr F $ is topologically mixing. As a matter of fact, $ T $ is topologically exact in the sense that, for every 
	open nonempty set $ U \subset \mathbb T $, there is $ M \ge 1 $ for which $ T^M (U) = \mathbb T $.
				 
        By a \emph{modulus of continuity}, we mean a continuous and non-decreasing function $\omega:[0,+\infty)\to[0,+\infty)$ with $ \omega(0) = 0 $.
        Let $\mathcal M$ be the set of all concave modulus of continuity. 
	For  $\omega\in \mathcal M$, we denote by $\mathscr{C}_{\omega}(\mathbb T)$ 
	the linear space of functions~$\varphi:\mathbb T\to~\mathbb R$ 
				with a multiple of $\omega$ as modulus of continuity:
        $ |\varphi(x)-\varphi(y)|\le C \omega(d(x,y)) $ for some constant $ C > 0 $, for all $ x,y\in [0,1) $.	
        For  $\varphi \in \mathscr C_\omega(\mathbb T)$, we denote the smallest constant that guarantees this condition of regularity by
$$\vert \varphi \vert_\omega :=\sup_{x\neq y}\frac{\vert\varphi(x)-\varphi(y)\vert}{\omega(d(x,y))}.$$

A central notion in this work will be the $T$-compatibility of a concave modulus of continuity with respect to another one.

\begin{definition}\label{modulosepreimagens}
For a map $T \in \mathscr F $, given $ \omega, \Omega \in \mathcal M $, we say that $ \Omega $ is $T$-compatible with respect to $ \omega $ whenever the following property
holds. 

\medskip

\noindent There are constants $\varrho_1>0$ and $C_1>0$ such that, given any sequence $\{x_k\}_{k\ge 0} $ in $\mathbb T$ with $T(x_{k+1})=x_k$ for $k\ge 0$, and a 
	point $y_0\in \mathbb T $ with $d(x_0,y_0) < \varrho_1$, there is a unique pre-orbit $\{y_k\}_{k\ge 1} $ of $ y_0 $ (that is, $T(y_{k+1})=y_k$ for $k\ge 0$) fulfilling
	\begin{align*}
	& d(x_k,y_k) \le d(x_0,y_0) < \varrho_1 \qquad \text{and}  \\
	& \Omega\big(d(x_k,y_k)\big)+ C_1\, \sum_{j=1}^k \omega\big(d(x_j,y_j)\big) \le \Omega\big(d(x_0,y_0)\big) \qquad \forall\, k\ge 1.
	\end{align*}
	Moreover, this correspondence between pre-orbits of $ x_0 $ and $ y_0 $ is one-to-one.
\end{definition}

We will show (see Proposition~\ref{compatibilidadesuficiente}) that a sufficient condition for $T$-compatibility of $\Omega$ with respect to $\omega$ is 
\begin{equation}\label{condicaosuficientecompatibilidade}
 \liminf_{x\to0} \frac{V(x)}{\omega(x)} \Big ( \Omega \big( (1+c) x \big) - \Omega (x) \Big) > 0 
\end{equation} 
for any sufficiently small constant $ c > 0 $. Note, in particular, that such a property implies  $ \liminf_{x\to 0}\frac{\omega(x)}{V(x)}=0 $. 

This condition allows us to easily provide families of examples of $T$-compatibility.  
By way of illustration, suppose that $ V $ and $ V^2 $ are concave near to the origin. 
If we consider $ \omega = V^2 $ and $ \Omega = V $ in this neighborhood,
then it is clear that condition~\eqref{condicaosuficientecompatibilidade}     
holds whenever $ \sigma > 0 $.
     				
For a more concrete situation, remember that a prototypical example in 
$\mathscr F$ is the Manneville-Pomeau map defined, for a fixed $s\in(0,1)$, as $ T_s(x):= x(1 + x^s) \mod 1$.
Consider the class of modulus of continuity $\omega_{\alpha,\beta}:[0,+\infty)\to [0,+\infty)$, defined for
$0\le\alpha<1$ and $\beta\ge 0$ with $ \alpha + \beta > 0 $ as
\begin{equation*}	
\omega_{\alpha,\beta}(x):=\left\{
\begin{array}{ll}
x^\alpha(-\log x)^{-\beta},&0< x < x_0,  \\
x_0^\alpha(-\log x_0)^{-\beta},& x \ge x_0,
\end{array}
\right.
\end{equation*}
where $x_0 = x_0(\alpha, \beta) $ is taken small enough so that $\omega_{\alpha,\beta}$ is concave.
This class was taken into account in the work of Kloeckner \cite{Klo20}.
While $\omega_{\alpha,0}$ is reduced to the H\"older continuity, $\omega_{0,\beta}$ determines a class that is 	
larger than local H\"older continuity. 
Note that for $x$ sufficiently small,
\begin{equation}\label{desigualdadealfabeta}
\omega_{\alpha,\beta}\big( (1+c) x \big) \ge (1+c)^{\alpha} \omega_{\alpha,\beta}(x).
\end{equation}
From this inequality, it follows that, for all $ s \in (0, \alpha) $, the modulus $ \omega_{\alpha - s, \beta} $ is 
$T_s$-compatible with respect to $ \omega_{\alpha, \beta} $ since condition~\eqref{condicaosuficientecompatibilidade} is immediately checked.
        
To provide an application example for the slowly varying scenario, consider the family of maps $ S_k(x) = x (1 + W_k(x))  \mod 1$, with 
$ k $ a positive integer, where in a neighborhood of the origin $ W_k(x) $ is of the form $ A_k (\log^k 1/x)^{-1} $ for some constant $ A_k > 0 $.  
(Here $ \log^k $ stands for the $k$-times composition of the logarithm function.) Suppose that $ \omega(x) $ and $ \Omega(x) $ are defined,
respectively, as $ (\log^k 1/x)^{-1} (\log 1/x)^{-1} (\log^2 1/x)^{-2} $ and $ (\log^2 1/x)^{-1} $ in a small neighborhood of the origin so that both
are concave. From the calculus exercise
$$ \lim_{x\to 0} \log 1/x \, \log \Big ( \frac{\log 1/x}{\log 1/(1+c)x} \Big ) = \log(1+c), $$
one may verify that condition~\eqref{condicaosuficientecompatibilidade} holds, and therefore the $S_k$-compatibility of (up to some convenient truncation on the right)  
the modulus $ (\log^2 1/x)^{-1} $ with respect to $ (\log^k 1/x)^{-1} (\log 1/x)^{-1} (\log^2 1/x)^{-2} $.

\subsection{Thermodynamics and Main Results}

Let $f$ be a  real continuous map with modulus of continuity $\omega\in \mathcal M$.  We define the \emph{transfer operator} associated with $f$ as 
$$\mathscr L_f\phi(x):=\sum_{y\in T^{-1}(x)}e^{f(y)} \phi(y), \quad \quad \forall \; \phi\in  C^0(\mathbb T),$$
where $C^0(\mathbb T)$ denotes the linear space of continuous functions endowed with the uniform norm $\vert\vert \cdot \vert  \vert_\infty. $ 
We have that $\mathscr L_f$ is a bounded linear operator. 
For every $n\ge 1$ and $x\in \mathbb T$, consider the Birkhoff sum
$S_n f(x):=\sum_{j=0}^{n-1} f\circ T^j(x).$ Then, clearly
$$\mathscr L_f^n\phi(x)= \sum_{y\in T^{-n}(x)}e^{S_nf(y)} \phi(y).$$
Let $\mathscr L_f^*$  denote the operator on finite signed Borel measures defined by
	$$
	\int \phi \, d(\mathscr L_f^*m)=\int \mathscr L_f \phi \,dm, \quad \forall \, \phi \in C^0(\mathbb T).
	$$
In other terms,   $\mathscr L_f^*$ is the \emph{dual operator} of $\mathscr L_f$.
Here we focus on its restriction on the convex subset  $ \text{Prob}(\mathbb T) $ of Borel probability measures on $\mathbb T$.
Let us represent by $ M(\mathbb T,T)$  the space of all $T$-invariant probability measures on $\mathbb T$.  
For a probability $m\in M(\mathbb T,T)$, we denote by $h_m(T)$ its metric entropy. 
For a continuous potential $f: \mathbb T\to \mathbb R$, we introduce by means of the variational principle the topological pressure as 
\begin{equation}\label{principio variacional}
P(T,f):=  \sup_{m\in M(\mathbb T,T)}\left\{h_m(T)+\int f \, dm\right\}.
\end{equation}

\medskip

Our central result states that, whenever $T$-compatibility can be verified, a Ruelle-Perron-Frobenius theorem holds. 

 \begin{theorem}\label{t:Theorem 1}
 	Let $T:\mathbb T\to \mathbb T$ be a map in $\mathscr F$ such that $T(x)=x(1+V(x))\mod 1$. 
	Let $ \Omega \in \mathcal M $ be a $T$-compatible modulus of continuity with respect to $\omega \in \mathcal M$. 
	If $f\in \mathscr C_{\omega}(\mathbb T)$, there exists a probability measure $\nu \in \text{Prob}(\mathbb T)$  and a positive constant  $\chi$  such that 
 	$$\mathscr L^*_f\nu= \chi \nu.$$
	The number $\chi$ is a simple eigenvalue of the operator $\mathscr L_f$ and there is a positive function $h\in \mathscr C_\Omega(\mathbb T)$ such that 
		$$\mathscr L_f h= \chi h.$$ 
       	The constant $ \chi $ is a maximal eigenvalue in the sense that the $ \mathscr L_f  $ acting on complex-valued continuous functions does not admit as eigenvalue another constant of absolute value greater than or equal to $ \chi $. 
       Moreover, supposing that $\int h \, d\nu=1$, for every continuous function $\phi,$ the sequence $ \{\chi^{-n}\mathscr L_f^n\phi\} $ converges uniformly on $\mathbb T$ to $h \int \phi \, d\nu$ as $n$ goes to infinity. 
 \end{theorem}
 
We are also able to establish the existence and uniqueness of Gibbs-equilibrium states.
 
 \begin{theorem}\label{t:Theorem 2}
In the context of Theorem~\ref{t:Theorem 1}, the measure $\mu := h\nu $ is a $T$-invariant probability such that 
		$$ h_{\mu}(T)+ \int f  \, d\mu = \log \chi = P(T,f). $$ 
For any $m\in  M(\mathbb T,T)$ with $m \neq \mu$, one has $P(T,f)>h_{m}(T)+ \int f \, dm$.
In particular, the probability $ \mu $ is the unique equilibrium measure associated with $ f $.
Furthermore, $ \mu $ is a Gibbs measure in the sense that, for every sufficiently small $ r > 0 $, there is a constant $ K_r > 0 $ such that,  
for  $ x\in \mathbb T$ and $ n \ge 1$,
		$$K_r^{-1}\le \frac{\mu\big(  \{y : d(T^j(x), T^j(y)) < r, \,\, 0 \le j \le  n\} \big)}{e^{S_n f(x)-nP(T, f)}}\le K_r.$$
 \end{theorem}
 
As one may expect, already known results can be seen as particular cases of the above theorems. By way of illustration, we recover as a corollary a 
result for the Manneville Pomeau family, studied, for instance, by \cite[Theorem~A]{Klo20} and \cite[Corollary~2.5]{LR14}, who considered for a potential 
 $f\in C_{\omega_{\alpha,0}}(\mathbb T)$ (H\"older modulus of continuity), the transfer operator $\mathscr L_f$ acting on $\mathscr C_{\omega_{\alpha-s},0}(\mathbb T)$.
  
\begin{corollary}\label{corA}
	For $s\in(0,1)$, consider the Manneville-Pomeau map $T_s(x)=x+x^{s+1}\mod 1$. Whenever $ 0 < s < \alpha < 1$ and $ \beta \ge 0$, for any potential  
	$f\in \mathscr C_{\omega_{\alpha,\beta}}(\mathbb T)$, the transfer operator $\mathscr L_f$ acting on $\mathscr C_{\omega_{\alpha-s,\beta}}(\mathbb T)$ 
	satisfies the Ruelle-Perron-Frobenius theorem. Furthermore, the invariant probability arising from the corresponding eigenfunction and eigenmeasure is the unique 
	Gibbs-equilibrium state associated with $ f $.
	\end{corollary}

Nor is it surprising that, due to the malleability of the framework considered, it is not difficult to present families of new examples. As far as we know, the result below is not registered in the literature. The reader may produce several others from condition~\eqref{condicaosuficientecompatibilidade}.

\begin{corollary}\label{corB}
Given a positive integer $ k $, let $ S_k(x) = x (1 + W_k(x))  \mod 1 $ be an element of $ \mathscr F $ 
for which $ W_k(x) = A_k (\log^k 1/x)^{-1} $, with $ A_k > 0 $, for any $ x $ small enough.  
Let $ f:\mathbb T \to \mathbb R  $ be a continuous potential with a positive multiple of $ (\log^k 1/x)^{-1} (\log 1/x)^{-1} (\log^2 1/x)^{-2} $ as modulus.
Then a Ruelle-Perron-Frobenius theorem holds when one considers the action of the associated transfer operator $\mathscr L_f$ 
on the linear space of the continuous functions that admit a positive multiple of $ (\log^2 1/x)^{-1} $ as a modulus of continuity.
In addtion, existence and uniqueness of Gibbs-equilibrium state associated with $ f $ are also guaranteed.
\end{corollary}

In section~\ref{Suficiente} we show that condition~\eqref{condicaosuficientecompatibilidade} 
is sufficient to guarantee $T$-compatibility and we provide examples of associated moduli of continuity 
that satisfy this condition. Sections~\ref{primeiroteorema} and~\ref{segundoteorema} are devoted to the proofs of the theorems above. 
For both results, the proof strategies fall on argumentative lines already present in the literature.
Among key references are \cite{Bal00, Bow75,   PP90, VO16, Rue04}.
A major contribution here is the identification of the linear space $ \mathscr C_\Omega(\mathbb T) $ as an appropriate set for the search of eigenfunctions of the transfer operator 
$ \mathscr L_f $ when $ f \in \mathscr C_\omega(\mathbb T) $.

\section{A sufficient condition for $T$-compatibility}\label{Suficiente}

By its very definition, it follows that $ T \in \mathscr F $ is expanding outside any half-closed arc that does not contain the origin, or without loss of generality 
outside any subset of the form $ [0, \epsilon) $, $0 < \epsilon < 1$.
Indeed, as $ T $ has exactly $ N_V := 1 + V(1) $ inverse branches, let $ \varrho_V \in (0, 1/2) $ be such that $ | x - y | < \varrho_V $ implies
$ |x - y| N_V + |V(x) - V(y)| < 1/2 $. It is thus easy to show that, for all $x,y\in [\epsilon, 1)$ with $ d(x, y) < \varrho_V $, 
          \begin{equation*}
          d(T(x),T(y))\ge \lambda(\epsilon)\, d(x,y),
          \end{equation*}
where  $\lambda(\epsilon) := 1 + V(\epsilon) \to 1$ as $\epsilon \to 0.$ 
A quantitative version of its non-uniformly expanding property on the whole circle is provided by the following lemma.
This result is analogous to Lemma~4 of \cite{GI20} and its proof is included by convenience of the reader. 

\begin{lemma}\label{desigualdade distancia}
	There exists a constant $\varrho_0 > 0$ such that for  $x,y\in \mathbb T$ with $d(x,y)<\varrho_0$, 
	\begin{equation*}
	d(T(x),T(y))\ge d(x,y) \, \Big(1+ \frac{1}{2^{\sigma+2}} V(d(x,y))\Big).
	\end{equation*}
\end{lemma}
\begin{proof}
 Let $x,y \in \mathbb T$ be such that $d(x,y)<\varrho_V$, where $ \varrho_V $ is defined as above. 
	We consider two situations.
	
	\smallskip
	\noindent\emph{Either the (smallest) open arc from $ y $ to $ x $ does not contain the origin.} We may suppose then $ 0 \le y \le x \le 1 $.
	Hence $ x(1+V(x)) - y(1+V(y)) \le |x - y| N_V + |V(x) - V(y)| < 1/2 $ and the fact that $ V $ is increasing imply that
	\begin{equation*}
	d(T(x),T(y)) = (x-y)\big(1+V(x)\big) + y\big(V(x) - V(y)\big) \ge d(x,y)\,\big(1+ V(d(x,y))\big).
	\end{equation*}
	
	\smallskip
	\noindent\emph{Or the open arc from $ y $ to $ x $ contains the origin.} By the previous case, we have
	\begin{equation*}
		d(T(x),0) \ge  d(x,0) \, \big(1+ V(d(x,0))\big) \quad \text{and} \quad d(0,T(y)) \ge  d(0,y) \, \big(1+ V(d(y,0))\big),
	\end{equation*}
	which by adding yields
	\begin{equation*}
		d(T(x),T(y)) \ge  d(x,y) + d(x,0) \, V(d(x,0)) + d(y,0) \, V(d(0,y)).
	\end{equation*}
	However, as $ V $ is increasing,
	\begin{align*}
	d(x,0) \, V(d(x,0)) + & d(y,0) \, V(d(0,y)) \ge \\
	& \ge \max\{d(x,0),d(y,0)\} \, V\big(\max\{d(x,0),d(y,0)\}\big) \\
	& \ge \frac{1}{2} d(x,y) \, V\Big(\frac{1}{2} d(x,y)\Big).
	\end{align*}
	Using the fact that $ V $ has a varying property, let $ \varrho_0 \in (0, \varrho_V) $ be such that 
	$ V \big(\frac{1}{2} \gamma \big) \ge \frac{1}{2^{\sigma+1}} V(\gamma) $ for $ 0 \le \gamma < \varrho_0 $. 
	We have thus shown that, whenever $ d(x,y) < \varrho_0 $,
	\begin{equation*}
	d(T(x),T(y))\ge d(x,y) \, \Big(1+ \frac{1}{2^{\sigma+2}} V(d(x,y))\Big).
	\end{equation*}
\end{proof}

The next result allows to check $T$-compatibility in concrete examples.

         \begin{proposition}\label{compatibilidadesuficiente}
         For a map $ T(x) = x(1+V(x)) \mod 1$  in $ \mathscr F $, suppose that the moduli of continuity $ \omega, \Omega \in \mathcal M $ fulfill
        $$
 \liminf_{x\to0} \frac{V(x)}{\omega(x)} \Big ( \Omega \big( (1+c) x \big) - \Omega (x) \Big) > 0 
 $$
for $ c > 0 $ small enough.  Then $ \Omega $ is $T$-compatible with respect to $\omega$.
          \end{proposition}

\begin{proof} 
  We shall show that there are $\varrho_1 \in (0, \varrho_0) $ and $C_1>0$ such that, for any $ x_1, x_0 $ in $\mathbb T$ with $T(x_1)=x_0$ and all
  $y_0\in \mathbb T $ with $d(x_0,y_0)< \varrho_1$, one has a unique $y_1\in \mathbb T $, with $T(y_1)=y_0$ and $d(x_1,y_1) \le d(x_0,y_0) < \varrho_1$, satisfying
	\begin{equation}\label{modulos apos uma iterada}
	\Omega\big(d(x_1,y_1)\big)+ C_1\, \omega\big(d(x_1,y_1)\big) \le \Omega\big(d(x_0,y_0)\big).
	\end{equation}
	Moreover, we shall argue that the correspondence $ x_1 \mapsto y_1 $ is injective.
	
	For a fixed $ c \in (0, \frac{1}{2^{\sigma + 2}}] $ such that the above limit inferior is positive, define 
	$ C_1 := \frac{1}{2} \liminf \frac{V(x)}{\omega(x)} \big ( \Omega ( (1+c) x ) - \Omega (x) \big) $. 
	Let $ \varrho_1 \le \varrho_0/2 $ be such that $ V(x) \in [0,1] $ and $ V(x)\big ( \Omega ( (1+c) x ) - \Omega (x) \big) \ge C_1 \omega(x) $ whenever $ 0 < x < \varrho_1 $, 
	where $\varrho_0$ is as in the statement of Lemma~\ref{desigualdade distancia}.
	For $x_0, x_1,  y_0 \in  \mathbb T $  with  $T(x_1)=x_0$ and  $d(x_0,y_0)<\varrho_1$, we can choose  $y_1\in T^{-1}(y_0)$ with
	$d(x_1,y_1) \le d(x_0,y_0) < \varrho_1$. 
	Then from Lemma~\ref{desigualdade distancia},
	\begin{equation*}
	d(x_0,y_0)=d(T(x_1), T(y_1))\ge d(x_1,y_1)\, \big(1+ c\, V(d(x_1,y_1))\big).
	\end{equation*}
	Since $\Omega$ is non-decreasing, we have
	$
	\Omega\big(d(x_0,y_0)\big) \ge \Omega\big(d(x_1,y_1)\,\big (1+ c\, V(d(x_1,y_1)\big)\big).
	$  	
	For $\gamma =d(x_1,y_1)$, we can write $ \gamma (1+c\, V(\gamma)) = (1-V(\gamma))\,\gamma+V(\gamma)\,(1+c)\, \gamma $.
	As  $\Omega $ is concave, we see that
	\begin{align*}
	\Omega\big(\gamma\,(1+c\, V(\gamma))\big)&\ge
	(1-V(\gamma))\,\Omega(\gamma)+V(\gamma)\,\Omega\big((1+c)\,\gamma\big)\\
	& =
	\Omega(\gamma)+ V(\gamma)\,\Big (\Omega((1+c)\,\gamma)-\Omega(\gamma)\Big) \\
	& \ge
	\Omega(\gamma)+ C_1 \omega(\gamma).
	\end{align*}
	Thus, we have shown that, for $x_0,x_1$, $y_0\in [0,1)$ with $T(x_1)=x_0$ and $d(x_0,y_0)<\varrho_1$, there is $y_1\in T^{-1}(y_0)$ 
	with $d(x_1,y_1) < \varrho_1$ for which inequality~\eqref{modulos apos uma iterada} holds. It remains to argue that $ x_1 \mapsto y_1 $ is a well-defined one-to-one map. 
	But Lemma~\ref{desigualdade distancia} implies that two pre-images $ \bar z_1 \neq z_1 $ of a point $ z_0 $ must satisfy $ d(\bar z_1, z_1) \ge \varrho_0 $. 
	Hence, we have $ \min \{ d(x_1, \bar y_1), d(\bar x_1, y_1) \} \ge \varrho_0 - d(x_1, y_1) > \varrho_1 $ whenever $ \bar x_1 \neq x_1 $ and $ \bar y_1 \neq y_1 $ are
	pre-images, respectively, of $ x_0 $ and $ y_0 $, which completes the proof of the lemma.
\end{proof}

\subsubsection*{Associated moduli fulfilling condition~\eqref{condicaosuficientecompatibilidade}}

Condition~\eqref{condicaosuficientecompatibilidade} is so flexible that it is not surprising there may be different moduli $T$-compatible with a given modulus $\omega$.
(Actually, inequality~\eqref{desigualdadealfabeta} allows  to show, via condition~\eqref{condicaosuficientecompatibilidade}, that for any $ \bar \beta \in [0, \beta] $ the modulus
$ \omega_{\alpha-s, \bar \beta} $ is $T_s$-compatible with $\omega_{\alpha, \beta} $.)
The question of whether there is a possible, say, canonical choice naturally arises.
Far from providing an answer to this interesting question, we would like to describe how a specific modulus of continuity $ \Omega $  can be determined in certain situations.
This construction is based on a case considered in \cite{GI20}.

For a map $T$ in $\mathscr F$ described as $T(x)=x(1+V(x))\mod 1$, suppose there exists $\omega \in \mathcal M$ for which 
there are constants  $\xi_0>1$ and  $\eta_0\in (0,1)$ and a function $ c: (1, \xi_0] \to (1, +\infty) $ such that
        	\begin{equation}\label{desigualdadeH}        	
	\frac{\omega(\xi x)}{V(\xi x)}\ge c(\xi) \frac{\omega(x)}{V(x)}, \qquad \forall\, x \in (0,\eta_0), \, \forall \, \xi\in (1,\xi_0].
        	\end{equation}

Note, for instance, that for the Manneville-Pomeau map $  T_s $ (for which $ V(x) = x^s $), and the modulus of continuity $\omega_{\alpha,\beta} $, 
the above hypothesis follows immediately with $ c(\xi) = \xi^{\alpha -s} $: for $\gamma$ sufficiently small,
\begin{equation*}
\frac{\omega_{\alpha,\beta}(\xi x)}{(\xi x)^s}
\ge
\xi^{\alpha-s} \frac{x^\alpha(-\log x)^{-\beta}}{x^s}= \xi^{\alpha-s} \frac{\omega_{\alpha,\beta}(x)}{x^s}.
\end{equation*}
	
We can define a new modulus of continuity $\Omega$ in $\mathcal M$ by means of~\eqref{desigualdadeH}. 
In order to do that, fix a parameter $ \tau > 0 $. First, let
 $\vartheta_0:[0,\infty)\to [0,\infty)$ be the continuous function defined as
\begin{equation*}
\vartheta_0(x):=\left\{
\begin{array}{ll}
\frac{\omega(x)}{V(x)},& x>0,  \\
0,& x= 0,
\end{array}
\right.
\end{equation*}
and let $ \vartheta_1:[0,\infty)\to [0,\infty) $ be  the continuous non-decreasing function given as
\begin{equation*} 
\vartheta_1(x)=\left\{
\begin{array}{ll}
	\displaystyle \max_{0\le y\le x} \vartheta_0(y),& 0\le x\le \tau,\\
	 \displaystyle\max_{[0,\tau]}\vartheta_0, & x\ge \tau,
	\end{array}
	\right.
\end{equation*}
Denote then $\vartheta_1^*$  the \emph{concave conjugate Legendre transform} of $\vartheta_1$: 
\begin{equation*}
\vartheta_1^*(x)= \min_{y\in[0,\infty)}[xy-\vartheta_1(y)], \quad \forall\, x\ge 0.
\end{equation*}
It is easy to see that $\vartheta_1^*$ is concave, non-decreasing and continuous on $[0,\infty)$.
Moreover, since for any $ x > 0 $ and $ \epsilon > 0 $, $ \vartheta_1^* (x) \le \epsilon - \vartheta_1(\epsilon/x) \le \epsilon $, we conclude that $ \vartheta_1^* $ is bounded and non-positive.
Its  \emph{concave conjugate Legendre transform},
\begin{equation*}
\vartheta_1^{**}(x)= \min_{y\in [0,\infty)}[xy-\vartheta_1^*(y)], \quad \forall\,  x\ge 0,
\end{equation*}
is also a continuous concave non-decreasing function. Moreover
$\vartheta_0(x)\le \vartheta_1(x)\le \vartheta_1^{**}(x)$ for all $x\in [0,\tau]$.
Actually, $\vartheta_1^{**}$ is the \emph{smallest} concave function that lies above $\vartheta_1$ on~$[0,\tau].$
Note that $\vartheta_1^{**}(0) = - \max \vartheta_1^*.$
We have thus obtained a function $ \Omega:= \vartheta_1^{**} + \max \vartheta_1^* $ that belongs to $\mathcal M$.

As an illustration, note that for the Manneville-Pomeau map $ T_s $ and the modulus of continuity $ \omega_{\alpha, \beta} $, 
whenever $ 0 < s < \alpha< 1 $ and $ \beta \ge 0 $, we have $ \vartheta_0 = \vartheta_1 = \vartheta_1^{**} = \Omega = \omega_{\alpha-s,\beta} $ on $ [0, \tau] $
if we take $ \tau = \min\big\{x_0(\alpha,\beta), x_0(\alpha-s, \beta)\big\} $.

Perhaps it is (at least conceptually) meaningful to note that, when $ \tau $ is taken small enough, the modulus $ \Omega $ could be introduced,
up to some truncation, as the concave hull of function $ \frac{\omega}{V} $. Indeed, inequality~\eqref{desigualdadeH} implies that  $ \frac{\omega}{V} $ is increasing 
on $ (0, \eta_0) $, and therefore the step $ \vartheta_1 $ is unnecessary for $ \tau < \eta_0 $. In this case, instead of the double Legendre transform, one
could prefer to define $ \Omega $ using the infimum of all affine functions bounding $  \frac{\omega}{V} $ from above. 

To see that the modulus of continuity $ \Omega $ so obtained is $T$-compatible with respect to the initial modulus $ \omega $,
the reader can easily adapt the arguments from the proof of Proposition 3 of \cite{GI20}, whose core strategy is exactly to show that, for this constructed modulus $ \Omega $ 
and the pair $ \omega $ and $ V $ fulfilling~\eqref{desigualdadeH},  one always has the condition~\eqref{condicaosuficientecompatibilidade} checked.
Even though an assumption like~\eqref{desigualdadeH} may be interpreted as somewhat restrictive, its value here is in pointing out that there may be more appropriate 
or convenient  choices of $T$-compatible moduli.

	\section{Proof of Theorem~\ref{t:Theorem 1}}\label{primeiroteorema}

	 The statements of Theorem~\ref{t:Theorem 1} are obtained from Proposition~\ref{autoequacoes}, Proposition~\ref{final teorema inicial} and Proposition~\ref{convergence of iterates}. We recall that $T:\mathbb T\to \mathbb T$ is a map in $\mathscr F$ such that $T(x)=x(1+V(x))\mod 1$,   $\omega $ and $ \Omega $ are moduli of continuity in $\mathcal M$ such that $ \Omega $ is $T$-compatible with respect to $\omega$ (recall Definition~\ref{modulosepreimagens}), and $ f $ is a potential that 
belongs to $ \mathscr C_\omega (\mathbb T) $.
 
\subsubsection*{Eigenproperties}

       \begin{proposition}\label{autoequacoes}
       The transfer operator $ \mathscr L_f $ and its dual share a common positive eigenvalue $ \chi $. 
       This number is a simple eigenvalue for $ \mathscr L_f $, associated with which there is a positive eigenfunction $ h $ belonging to $ \mathscr C_\Omega (\mathbb T) $.
       \end{proposition}
   
       \begin{proof}
		Recall that  $ \text{Prob}(\mathbb T) $ denotes the  space of Borel probability measures on $ \mathbb T $ equipped with the weak-star topology
	and the dual operator $ \mathscr L_f^* $ acts on it
 as follows
	$$
	\int \phi \, d(\mathscr L_f^*\mu) = \int \mathscr L_f \phi \, d\mu, \qquad \forall \, \phi \in  C^0(\mathbb T).
	$$
	Let $ \mathds 1 $ denote the function identically equal to 1 on $ \mathbb T $.
	The function $ \Phi $ defined on $ \text{Prob}(\mathbb T) $ as
	 $$\Phi(\mu):= \frac{\mathscr L_f^*\mu}{\int \mathscr L_f \mathds 1 d\mu} $$
	 is clearly continuous.
	Since $ \text{Prob}(\mathbb T) $ is a convex and compact set  which is  invariant by  $\Phi$, then   
	Schauder-Tyckhonov theorem guarantees $\Phi$ admits a fixed point of $ \nu \in \text{Prob}(\mathbb T) $. 
	Hence, for $\chi := \int \mathscr L_f\mathds 1 \, d\nu>0$, we have 
	\begin{equation*}
	\mathscr L_f^*\nu= \chi\,\nu.
	\end{equation*}

	Let $\varrho_1 $ and $ C_1$ be the constants that characterize the $T$-compatibility of $\Omega$ with respect to $\omega$ (see Definition~\ref{modulosepreimagens}).	
	Given $f\in \mathscr C_\omega$, denote $\kappa_f:= C_1^{-1}\vert f\vert_\omega$. Consider thus the following subset of $C^0(\mathbb T)$:
	\begin{equation*}
	\Lambda:=\left\{\phi\in  C^{0}(\mathbb T):\,
	\phi\ge 0, \, \int \phi \,d\nu=1, \,\,\phi(x)\le \phi(y)\, e^{\kappa_f \, \Omega(d(x,y))} \textrm{ if } d(x,y)<\varrho_1\right\}.
	\end{equation*}
	We have that $\Lambda$ is a convex and closed nonempty subset of  $ C^0(\mathbb T)$. We claim that $\Lambda$ is uniformly bounded. 
	In fact, let $ \{A_i\}_{i=1}^L $ be a finite cover of $ \mathbb T $ by open arcs of length $ \varrho_1 $ and let $ z_i \in \mathbb T $ denote the center of $ A_i $.
	Note that we may always suppose that these points are positively oriented, that is, $ z_1 < z_2 < \ldots < z_L $.
	Hence, given $x,y\in \mathbb T$, with $  x < y  $,
	 consider indexes  $i_x \le i_y$ for which
	$x\in A_{i_x} $ and $ y \in A_{i_y} $, 
	so that 
	$d(x, z_{i_x})< \varrho_1/2$, $d(y, z_{i_y})< \varrho_1/2$, and for every $i_x\le i< i_y$,  $d(z_i,z_{i+1})< \varrho_1$. 
	For an arbitrary $\phi $ in $ \Lambda$, the local property in the definition of this set provides
 \begin{align}
\phi(x)\le &\phi(z_{i_x})e^{\kappa_f \Omega(d(x,z_{i_x}))}  \nonumber \\ 
\le & \phi(y)  \exp\Big(\kappa_f\big(\Omega(d(x,z_{i_x}))+\displaystyle\sum_{i= i_{x}}^{i_{y}-1} \Omega(d(z_i,z_{i+1})) + \Omega(d(z_{i_y},y))\big)\Big) \nonumber \\
\le &\phi(y) e^{L \, \kappa_f \, \Omega(d(x,y))}. \label{extensao propriedade local de Lambda}
\end{align}
As $0\le \min \phi \le \int \phi d\nu=1$, in particular for $x, y\in \mathbb T$ such that $ \phi(x) = \| \phi \|_\infty $ and $\phi(y) = \min \phi $,  it follows
	 $$\|\phi\|_\infty \le \min \phi  \, \, e^{L \, \kappa_f \,\Omega(1/2)} \le e^{L \, \kappa_f \,\Omega(1/2)}. $$
The above estimates also ensure that $\Lambda$ is equicontinuous. In fact, as $ | e^a - 1 | \le | a | \, e^{|a|} $, we have
	\begin{equation}\label{continuidade uniforme em Lambda}
	\vert\phi(x)-\phi(y) \vert  \le \| \phi \|_\infty \, \big| e^{L \, \kappa_f \, \Omega(d(x,y))} - 1 \big|  \le L \kappa_f e^{2 L \, \kappa_f \, \Omega(1/2)}  \, \Omega(d(x,y)),
	\end{equation}
	for $ \phi \in \Lambda $ and $ x, y \in \mathbb T $.
	Therefore, by Arzel\`a-Ascoli theorem, the set $\Lambda$ is compact.
	
	For $ \phi \in \Lambda $, define 
	$$\mathscr T(\phi):= \mathscr L_{f - \log \chi}\phi = \frac{1}{\chi}\mathscr L_f\phi= \frac{\mathscr L_f\phi}{\int \mathscr L_f \mathds 1 d\nu} \ge 0.$$ 
	The set $\Lambda$ is invariant under the operator $\mathscr T$, that is: $\mathscr T (\Lambda)\subseteq \Lambda.$ 
	Indeed,  $\mathscr T$ is clearly a positive operator and we note that for $\phi\in \Lambda$, 
	\begin{equation*}
	\int \mathscr T(\phi) \; d\nu=\int \frac{1}{\chi}\mathscr L_f\phi \; d\nu= \int \frac{\phi}{\chi} \; d(\mathscr L_f^*\nu)=
	\int \phi \; d\nu=1. 
	\end{equation*}
Recall that we denote $N_V= 1+V(1)$. For a pair of points $x,y\in \mathbb T$ with $ d(x, y) \le \varrho_1 $, if for $1\le i\le N_V$, 
$  x_i $ denotes a preimage of $ x $,  let  $ y_i $ be the corresponding  preimage of $ y $ as stated in Definition~\ref{modulosepreimagens}.  Then,
the fact that $\mathscr T(\phi)(x)\le \mathscr T(\phi)(y) \, e^{\kappa_f \, \Omega(d(x,y))}$ is a consequence of 
	\begin{align*}
	\mathscr L_f\phi(x)=&\sum_{i=1}^{N_V} e^{f(x_i)} \phi(x_i)
	\le\sum_{i=1}^{N_V} e^{f(x_i)} \phi(y_i)\, e^{\kappa_f\,\Omega(d(x_i,y_i))}\\
	\le&\sum_{i=1}^{N_V} e^{f(y_i)+ \vert f \vert_\omega\, \omega(d(x_i,y_i))} \phi(y_i)\, e^{\kappa_f\,\Omega(d(x_i,y_i))}\\
	=&\sum_{i=1}^{N_V} e^{f(y_i)} \phi(y_i)\, e^{\kappa_f
		\big( C_1 \omega(d(x_i,y_i))+ \Omega(d(x_i,y_i))\big)}\\
	\le 	&\sum_{i=1}^{N_V} e^{f(y_i)} \phi(y_i)\, e^{\kappa_f\,
		\Omega(d(x,y))}
	= 	\mathscr L_f \phi(y) \, e^{\kappa_f\,
		\Omega(d(x,y))}.
	\end{align*}
	(Note that for the last inequality we apply the $T$-compatibility of $ \Omega $ with respect to $ \omega $.)
	Applying the Schauder-Tychonoff theorem for $\mathscr T: \Lambda\to \Lambda$,  there is $h\in \Lambda$ such that 
	$\mathscr L_f h = \chi h.$ From~\eqref{continuidade uniforme em Lambda}, $ h \in \mathscr C_\Omega(\mathbb T) $.
	To show that $h>0$, we suppose by contradiction that $h(z)=0$  for some  $z\in \mathbb T$. Hence $ \chi^{-n}\mathscr L_f^n h(z)=0$  for every $n\ge 1$. Then
	for every $y\in T^{-n}(z)$ we have that $h(y)=0$. Since $T$ is topologically mixing,  
	the set $\bigcup_{n\ge 0} T^{-n}(z)$ is dense, which implies that $h=0$ in $\mathbb T$. But since $h\in \Lambda$, we have $\int h d\nu=1$, which is a contradiction.
	
        To prove that $\chi $ is a simple eigenvalue of the operator $\mathscr L_f $, 
	let
	$\phi$ be a continuous function such that $\mathscr L_f \phi = \chi \,\phi $.
	Since $\mathbb T$ is compact, there is $z\in \mathbb T$ such that 
	$$
	\min_{x\in \mathbb T} \frac{\phi(x)}{h(x)} = \frac{\phi(z)}{h(z)}.
	$$
	The function $\hat \phi:= \phi-\frac{\phi(z)}{h(z)} h$ is continuous. Moreover,  $\hat \phi$ verifies for every $n \ge 1$, 
	$$\mathscr L_f^n\hat\phi(z)= \mathscr L_f^n\phi(z)-\chi^n\, \phi(z)=0.$$
	Hence, since $\hat \phi$ is nonnegative, as above the fact that  $T$  is topologically mixing implies that
	$\hat \phi = 0 $ on $\mathbb T$, i.e.,  
	$\phi=\frac{\phi(z)}{h(z)} h$. Therefore,  every eigenfunction for $\chi $ is a multiple of $h$. 
	\end{proof}

	\subsubsection*{Iterates of the transfer operator}

	We focus on the behavior of iterates of the transfer operator $ \mathscr L_f $ and we derive in particular the maximal character of the eigenvalue $ \chi $. 
	Henceforward, eigenfunction $ h $ and eigenprobability $ \nu $ are supposed  to fulfill $ \int h \, d\nu = 1 $.

        \begin{lemma}
		For $\phi \in C^0(\mathbb T)$, the sequence $\big\{ \frac{1}{\chi^n} \mathscr L_f^n \phi \big\}_{n\ge 1}$ is uniformly equicontinuous and uniformly bounded.
	\end{lemma}
	
	\begin{proof}
		Let $x, y\in \mathbb T$ be such that $d(x,y)<\varrho_1$. 
		Given $ \{x_k\}_{k \ge 1} $ a pre-orbit of $x$, the $T$-compatibility of $\Omega$ with respect to $\omega$ 
		ensures that there exists a unique pre-orbit $ \{y_k\}_{k \ge 1} $ of $y$ such that
		$d(x_k, y_k) \le d(x,y) < \varrho_1$ and $ \Omega\big(d(x_k,y_k)\big)+ C_1\, \sum_{j=1}^{k} \omega\big(d(x_j,y_j)\big) \le \Omega\big(d(x,y)\big) $, for all $ k \ge 1 $.
		In particular, for $f\in \mathscr C_\omega(\mathbb T) $, we have the following estimates for the corresponding Birkhoff sums  
		\begin{align}\label{variation sum}
		\nonumber\vert S_n f(x_n)-S_n f(y_n) \vert \le & \vert f\vert_\omega \sum_{j=0}^{n-1}\omega(d(T^j(x_n), T^j(y_n))) \\
		\le &
		 \kappa_f \Big(\Omega(d(x,y))-\Omega(d(x_n,y_n))\Big)\le 
		\kappa_f \; \Omega(d(x,y)),
		\end{align}
		where, as before, $\kappa_f = C_1^{-1}\vert f\vert_\omega$. 
		
		Keeping the notation of pairs of pre-images $ (x_n, y_n) $ 
		associated by the the correspondence established by $T$-compatibility of $\Omega$ with respect to $\omega$, we can write for every $\phi\in C^0(\mathbb T) $
		$$ \Big\vert\mathscr L_f^n \phi (x)-\mathscr L_f^n \phi (y)\Big\vert \le 
		 \sum_{(x_n, y_n)} \Big\vert e^{S_nf(x_n)}\phi(x_n)- e^{S_nf(y_n)}\phi(y_n) \Big\vert $$
		If we denote $ \omega_{\phi}(t):= \sup\{ \vert \phi(x)-\phi(y)\vert: x, y\in \mathbb T, \,  d(x,y)\le t \} $, we easily obtain 
		\begin{align*}
		\Big\vert e^{S_nf(x_n)}\phi(x_n) - & e^{S_nf(y_n)}\phi(y_n) \Big\vert \le \\
		& \le e^{S_nf(x_n)} \, \omega_{\phi}(d(x_n,y_n)) + \vert\vert\phi\vert\vert_\infty \, \Big\vert e^{S_nf(x_n)}-e^{S_nf(y_n)}\Big\vert \\
		& \le e^{S_nf(x_n)} \, \omega_{\phi}(d(x,y)) + \vert\vert\phi\vert\vert_\infty \, e^{S_nf(y_n)} \, \Big\vert e^{S_nf(x_n)- S_nf(y_n)} - 1\Big\vert.
	  \end{align*}
		Using~\eqref{variation sum} we get 
		$$ \Big\vert e^{S_nf(x_n)- S_nf(y_n)} - 1\Big\vert \le \kappa_f e^{\kappa_f \; \Omega(1/2)} \; \Omega(d(x,y)). $$
		Therefore, we have shown that
		\begin{align}\label{desigualdade iteradas}
		\Big\vert \mathscr L_f^n \phi (x) & -  \mathscr L_f^n \phi (y)\Big\vert \le  \nonumber \\
		& \le \omega_{\phi}(d(x,y)) \, \mathscr L_f^n \mathds 1(x) + \vert\vert\phi\vert\vert_\infty \kappa_f e^{\kappa_f \; \Omega(1/2)} \Omega(d(x,y)) \, \mathscr L_f^n \mathds 1 (y).
		\end{align}
		Note now that, from~\eqref{extensao propriedade local de Lambda}, the positive eigenfunction $ h $ (as any other element of the set $ \Lambda $) 
		satisfies $\frac{h(T^n(w))}{h(w)}\le e^{L \,\kappa_f \, \Omega(1/2)} $, for all $n\ge 1$ and $w \in \mathbb T $.
		Hence, we see that
		\begin{align*}
	         \frac{1}{\chi^n} \mathscr L_f^n \mathds 1(z) = \frac{1}{\chi^n}\sum_{w \in T^{-n}(z)} e^{S_nf(w)} 
	         & \le \frac{1}{\chi^n} \sum_{w \in T^{-n}(z)} e^{S_nf(w)} \frac{h}{h\circ T^n}(w) e^{L \, \kappa_f \, \Omega(1/2)} \\
		       & = e^{L \, \kappa_f \, \Omega(1/2)}\frac{1}{\chi^n h(z)} \mathscr L_f^n h(z) = e^{L \, \kappa_f \, \Omega(1/2)}.
		\end{align*}
		From the above discussion, we deduce that
		\begin{align*} 
		\Big\vert \frac{1}{\chi^n} \mathscr L_f^n\phi(x)- & \frac{1}{\chi^n} \mathscr L_f^n\phi(y)\Big\vert \le \\
		& \le e^{L \, \kappa_f \, \Omega(1/2)} \Big (\omega_{\phi}(d(x,y))  + \vert\vert\phi\vert\vert_\infty \kappa_f e^{\kappa_f \; \Omega(1/2)} \Omega(d(x,y)) \Big),
		\end{align*}
		from which we conclude that $\big\{ \frac{1}{\chi^n} \mathscr L_f^n \phi \big\}$ is  uniformly equicontinuous. 
		Moreover, since
		$\Big\vert \frac{1}{\chi^n} \mathscr L_f^n \phi (x) \Big\vert  \le  
		\vert \vert \phi \vert \vert_{\infty} \frac{1}{\chi^n} \mathscr L_f^n \mathds 1 (x) \le \vert\vert \phi\vert\vert_\infty e^{L \, \kappa_f \, \Omega(1/2)} $,
		the sequence $\big\{ \frac{1}{\chi^n} \mathscr L_f^n \phi \big\}$ is also uniformly bounded.
	\end{proof}

	\begin{proposition}\label{convergence of iterates}
	The sequence $\big\{ \frac{1}{\chi^n} \mathscr L_f^n \phi \big\}_{n\ge 1}$ converges uniformly to $h\int \phi\,  d \nu$.
	\end{proposition}
	
	\begin{proof}
	Thanks to the previous lemma and Arzel\`a-Ascoli theorem, we need to argue that any convergent subsequence  $ \big\{\frac{1}{\chi^{n_j}}\mathscr L_f^{n_j} \phi \big \}  $ has as uniform limit $h\int \phi\,  d \nu$.
	Suppose that  $ \big \{\frac{1}{\chi^{n_j}}\mathscr L_f^{n_j} \phi \big \}  $ converges uniformly to $ \phi_\infty $.
	
	Consider the normalized potential $\tilde f= f+\log h-\log h\circ T-\log \chi $ and note that for all $n\ge 1$ and $\psi \in C^{0}(\mathbb T)$,  $\frac{1}{\chi^n}\mathscr L^n_{f}\psi = h\mathscr L^n_{\tilde f}\big(\frac{\psi}{h}\big) $.
	Since $ \mathscr L_{\tilde f} \psi \le \max \psi \, \mathscr L_{\tilde f} \mathds 1 = \max \psi $, note also that
	\begin{equation}\label{decrescimento operador}
	 \cdots \le \max \mathscr  L^n_{\tilde f} \psi \le \cdots \le  \max \mathscr L^2_{\tilde f} \psi \le \max \mathscr L_{\tilde f} \psi \le \max \psi.
	\end{equation}
	
	We thus have that $ \Big\{ \max \mathscr L_{\tilde f}^{n}\big(\frac{\phi}{h}\big)\Big\}_{n \ge 1}$ is non-increasing and 
	$ \Big\{ \mathscr L_{\tilde f}^{n_j} \big(\frac{\phi}{h}\big) \Big\}_{j \ge 1}  $ converges uniformly  $ \frac{\phi_\infty}{h} $, 
	so that, given $ \epsilon > 0 $, for $ j $ sufficiently large,  		
	$$ \max \mathscr L_{\tilde f}^{k} \Big(\frac{\phi_\infty}{h}\Big) \ge \max \mathscr L_{\tilde f}^{k}\Big(  \mathscr L_{\tilde f}^{n_j} \big(\frac{\phi}{h} \big) \Big) + \epsilon 
	\ge  \max \mathscr L_{\tilde f}^{n_{k+j}} \big(\frac{\phi}{h} \big) + \epsilon,$$
        for any fixed $ k $. By passing to the limit as $ j $ tends to infinity and then considering $ \epsilon > 0 $ arbitrarily small, from~\eqref{decrescimento operador} we conclude that
        \begin{equation}\label{igualdademaximos}
         \max \mathscr L_{\tilde f}^{k} \big(\frac{\phi_\infty}{h}\big) = \max \frac{\phi_\infty}{h} \qquad \forall \, k \ge 1. 
         \end{equation}
        As $ \mathscr L_{\tilde f}^k \mathds 1 = \mathds 1 $, it follows that 
		$$ T^{-k} \Big( \argmax \mathscr L_{\tilde f}^k \big(  \frac{\phi_\infty}{h} \big)  \Big) \subset \argmax \frac{\phi_\infty}{h} \qquad \forall \, k \ge 1.	$$
	Since $T$ is topologically mixing, we thus obtain that $\frac{\phi_\infty}{h}$ attains its maximum value  in any nonempty open set of $\mathds T$. 
	Hence, by continuity, $\frac{\phi_\infty}{h}$ is identically constant.
	
	Note now that by the dominated convergence theorem
	$$ \lim_{j \to \infty} \int \mathscr L_{\tilde f}^{n_j} \big(\frac{\phi}{h}\big) \, d\mu = \int  \frac{\phi_\infty}{h}  \, d\mu =  \frac{\phi_\infty}{h}. $$
	Since $ \int   \mathscr L^n_{\tilde f} \psi \, d\mu = \int \psi \, d\mu $ for all $n\ge 1$ and $\psi \in C^{0}(\mathbb T)$, we have shown that  $\phi_\infty = h\int \frac{\phi}{h} d\mu = h\int \phi\,  d \nu$.
	\end{proof}
		
	We can now complete the proof of Theorem~\ref{t:Theorem 1} by discussing the maximality of the eigenvalue $ \chi $.

        \begin{proposition}\label{final teorema inicial}
        When acting on complex-valued continuous functions on $ \mathbb T $, 
        the transfer operator $  \mathscr L_f $ does not possess another eigenvalue with an absolute value strictly greater than or equal to $ \chi $.
        \end{proposition}
        
        \begin{proof}
        Note that $  \mathscr L_f  \phi = c \phi $ for a (non null) complex-valued continuous function $ \phi $ and a constant $ | c | \ge \chi $ if, and only if, 
        $  \mathscr L_{\tilde f} \left( \frac{\phi}{h} \right) = \frac{c}{\chi} \frac{\phi}{h} $, where as before $\tilde f= f+\log h-\log h\circ T-\log \chi $. Thus, it suffices to show that  
        $  \mathscr L_{\tilde f} $ acting on complex-valued continuous functions admits only 1 as eigenvalue outside the open unit disc. 
         Suppose then $  \mathscr L_{\tilde f} \phi = c \phi $ with $ | c | \ge 1 $. Clearly, $ |\phi| \le  \mathscr L_{\tilde f}^k  |\phi| $ for all $ k \ge 1 $, so that 
         $ \max |\phi| \le \max \mathscr L_{\tilde f}^k  |\phi| \le \max |\phi| \max \mathscr L_{\tilde f}^k \mathds 1 = \max  |\phi| $. 
         We are exactly in the same situation as~\eqref{igualdademaximos}. Therefore, we obtain that $ |\phi| $ is constant, which we may assume equal to 1.
         Hence, we write $ \phi(x) = e^{2 \pi i \theta(x)} $ and $ c = b e^{2 \pi i \gamma} $ with $ b \ge 1 $ and $ \gamma \in \mathbb R $. 
         Since $  \mathscr L_{\tilde f} e^{2 \pi i \theta} = b e^{2 \pi i (\theta + \gamma)} $ represents a convex combination of extremal points of the unit disc,
         we conclude that $ b = 1 $ and $ \theta(y) = \theta(x) + \gamma \mod 1 $ for all $ x \in \mathbb T $ and $ y \in T^{-1}(x) $. In particular, for $ x=y=0 $
         we see that $ \gamma \in \mathbb Z $, and therefore $ c = 1 $.
         \end{proof}

	\section{Proof of Theorem~\ref{t:Theorem 2}}\label{segundoteorema}
	
	In this section, we discuss a succession of intermediate results, from which we will derive Theorem~\ref{t:Theorem 2}.
	Throughout the entire section, we will assume without mentioning the hypotheses of Theorem~\ref{t:Theorem 1}.
	Moreover, the positive eigenfunction $ h $  and the eigenprobability $ \nu $, both obtained in Theorem~\ref{t:Theorem 1}, 
	are from now on supposed to be related as $ \int h \; d\nu = 1 $.
	The statements of Theorem~\ref{t:Theorem 2} can be recovered from the statements of Lemma~\ref{lem: lem 4}, Proposition~\ref{Conditions 1 2 and 3} 
		 and Proposition~\ref{Prop1}.

		\subsubsection*{Equilibrium states from Rokhlin formula}
	
	For topological dynamical systems, whenever the measure entropy of $ T $ is upper semi-continuous with respect to the measure,
	one may guarantee the existence of an invariant probability attaining the supremum in the variational expression~\eqref{principio variacional}
	of the topological pressure, namely, the existence of an \emph{equilibrium state}. Moreover, if the topological entropy of the system is finite, 
	then the extreme points of the convex set of equilibrium states are exactly the ergodic members of this set.  See, for instance,~\cite[Theorem 9.13]{Wal82}.
	The fact that, for the maps we are dealing with, the measure entropy, regarded as a function of the measure, is upper semi-continuous follows from a 
	general result for piecewise monotone mappings of the circle~\cite[Corollary 2']{MS80}.
	
Our aim now is to show that $ \mu = h \nu $ is the unique equilibrium state associated with $ f $.
A key element in our argument will be Rokhlin formula for the measure entropy. We briefly recall the main ingredients.

Let $ m $ denote a Borel probability measure on $ \mathbb T $.
Let $ \{\mathscr A_n\}_{n \ge 1} $ be a sequence of measurable (countable) partitions of $ \mathbb T $, with finite $m$-entropy, such that $\mathscr A_n \preceq \mathscr A_{n+1}$ for all $n $.
We say that $ \{\mathscr A_n\}_{n \ge 1} $ is   \emph{$m$-generating}  if $ \bigcup_{n \ge 1} \mathscr A_n $ generates the Borel $\sigma$-algebra, up to $m$-measure zero.
For $m$-almost every $ x \in \mathbb T $, we denote $ \mathscr A_n(x) $ the element of the partition $ \mathscr A_n $ to which $ x $ belongs.
A sufficient condition for  $ \{\mathscr A_n\}_{n \ge 1} $ to be $m$-generating is to satisfy
\begin{equation}\label{diametro e gerador}
\Dia(\mathscr A_n(x)) \to 0 \quad \text{ as } n \to \infty, \qquad \text{ for $m$-a.e. } x \in \mathbb T.
\end{equation}
(For a proof of this fact, see, for instance, the proof of Corollary~9.2.8 in~\cite{VO16}.)	

Given a map $ T $ in $ \mathscr F$, a measurable function  $J_m(T): \mathbb T\to [0,\infty)$   is a  \emph{Jacobian}  of $T$ with respect to $ m $  if  for any  measurable  set $A$ such that $T\vert_A$ is injective, 
$$
m(T(A))= \int_A J_m(T) \, dm.
$$
Whenever $m$ is a  $T$-invariant probability measure, existence and uniqueness (up to $m$-measure zero) of a Jacobian of $T$ with respect to $m$ are well known. 
(For a more general result, see~\cite[Proposition 9.7.2]{VO16}.) 
For a $T$-invariant probability $ m $, it is easy to see that  $J_m(T)>0 $ $m$-almost everywhere. 
Moreover,
	\begin{equation}\label{Property Jacobian}
	\sum_{y\in T^{-1}(x)} \frac{1}{J_m(T)(y)}=1 \quad \textrm{ for } m\textrm{-almost every } x\in \mathbb T.
	\end{equation}
The following formula, due to V. Rokhlin, allow us to compute the entropy from the Jacobian. 
For a proof of this classical result, see, for instance, \cite{Par69}.

\begin{teorema}[Rokhlin formula] 
	Let $T$ be a locally invertible measurable transformation and $m$ be a $T$-invariant probability measure. 
	Suppose that domains of invertibility of $ T $ provide a partition $ \mathscr A_0 $ such that the sequence  
	$ \big\{ \vee_{j=1}^n T^{-j}(\mathscr A_0)\big\}_{n \ge 1} $ is $m$-generating. Then
	$$
	h_m(T) = \int \log J_m (T) \, dm.
	$$
\end{teorema}

An application of the previous facts will be summarized in the following lemma.

\begin{lemma}\label{Partition and Rokhlin formula}
	For $T$ a map in $\mathscr F$, let  $m$  be a  $T$-invariant probability that does not charge $ 0 $. 
	If the set of pre-images $\{a_i\}_{i=0}^{N_V-1}$ of $ a_0 := 0 =: a_{N_V} $ is supposed to be positively oriented, 
	let $ A_i $ be the positively oriented open arc from $ a_i $ to $ a_{i+1} $. 
	Denote  $\mathscr A:=\{A_i\}_{i=0}^{N_V-1}$ and $ \mathscr A_n: = \bigvee_{j=0}^{n-1} T^{-j}(\mathscr A) $. 
	Then $\{\mathscr A_n\}_{n\ge 1}$ is $m$-generating. 
	In particular, the measure $m$ satisfies the Rokhlin formula for the entropy
	$h_m(T)$ and the Jacobian $J_m(T)$.
\end{lemma}

\begin{proof}	
	Obviously by invariance $ m $ does not charge any point $ a_i $, and therefore $ \{A_i\}_{i=0}^{N_V-1} $ is a partition of $\mathbb T$ with respect to $ m $. 
	To prove that the monotone sequence $\{\mathscr A_n\}$ is $m$-generating we will show that~\eqref{diametro e gerador} holds. 
	Actually, this property follows easily from the fact that $T$ is topologically exact. 
	Indeed,  an element $\mathscr A_n(x) $ of $\mathscr A_n$ is of the form
	$$
	\mathscr A_n(x)= \bigcap_{j=0}^{n-1} T^{-j}(A_{i_j}),
	$$
	where $A_{i_j}$ is the open arc from $a_{i_j}$ to $a_{i_j+1}$, $i_j\in\{0,1,\cdots, N_V-1\}$.
	Now if the diameters would not shrink for a particular $ x \in \mathbb T \setminus \bigcap_{j \ge 0} T^{-j}\big(\{a_0, \ldots, a_{N_V-1}\}\big) $,
	then there would exist~$\kappa>0$ and a sequence $ \{n_j\}_{j\ge 0} $, with  $n_j\to +\infty$ as $j\to +\infty$, 
	such that for every $j\ge 0 $,
	$$\Dia (\mathscr A_{n_j}(x))\ge \kappa.$$
	Hence, for an open nonempty subset $U\subset \displaystyle \bigcap_{j=0}^{+\infty} \mathscr A_{n_j}(x)$, we would have
	$$T^k(U)\subset  T^k \big ( \bigcap_{j=0}^{+\infty}  \mathscr A_{n_j}(x) \big) \subset   A_{i_k}, \quad \forall \, k\ge 1.$$
	However, there exists $M>0$ such that 
	$$\mathbb T=T^M(U) \subset   A_{i_M} ,$$ 
	which is a contradiction. Thus, property~\eqref{diametro e gerador} holds and the Rokhlin formula can be applied to the probability $ m $.
\end{proof}

Our first goal is to show that the measure obtained from the eigenfunction of the transfer operator and the eigenprobability of the dual operator 
satisfies the conditions of the preceding lemma.

	\begin{lemma} \label{lem: lem 4} For a map $ T $ in $\mathscr F$, let $\{a_i\}_{i=1}^{N_V-1}$  denote the points of $\mathbb T \setminus \{0\} $ such that $T(a_i)=0$.	
	Then, $\mu= h\nu $ is a $T$-invariant probability that does not charge either $0$ or any  $a_i$, $i=1,\cdots, N_V-1$. Furthermore, the Jacobian of $T$ with respect to $ \mu $ 
	is given as $J_\mu(T)= \chi \, \frac{h\circ T}{h} \, e^{-f} $.
	\end{lemma}
	
	\begin{proof}
		Consider once more the normalized potential $\tilde f= f+\log h-\log h\circ T-\log \chi$ and the associated transfer operator $ \mathscr L_{\tilde f} $.
		The invariant property of $\mu$ follows thus immediately: for all $\psi$ in $C^0(\mathbb T)$,
		\begin{equation*}
		\int\psi\circ T\, d\mu= \int\psi\circ T\, d(\mathscr L_{\tilde f}^*\mu)
		= \int \mathscr L_{\tilde f}(\psi\circ T)\, d\mu= \int \psi\mathscr L_{\tilde f}\mathds 1 \, d\mu=\int \psi\, d\mu.
		\end{equation*}
		
		Note now that, by this invariant property, 
		\begin{equation*}
		\mu(\{0\})= \mu(T^{-1}(0))= \mu(\{0\})+\sum_{i=1}^{N_V-1}\mu(\{a_i\}),
		\end{equation*}
		which implies $\mu(\{a_i\})=0$ for $i=1,\cdots, N_V-1$.
		We also note that $\mu(\{0\})=0$. Otherwise, if we suppose  $\mu(\{0\})>0$, we would have for $\psi$ in $C^0(\mathbb T)$,
		\begin{align*}
		\frac{1}{\chi h(0)}\mathscr L_f(h\psi)(0) \, \mu(\{0\})  
		& =  \int_{\{0\}} \mathscr L_{\tilde f} \psi \, d\mu =  \int \mathscr L_{\tilde f} (\mathds 1_{T^{-1}(0)} \psi) \, d\mu \\
		& = \int_{T^{-1}(0)} \psi \, d\mu = \psi(0)\,\mu(\{0\}).
		\end{align*}
		(Here $ \mathds 1_{T^{-1}(0)} $ represents the indicator function on the set of pre-images of 0.)
		Hence, the following (linear) equation would hold for every  $\psi$ in $C^0(\mathbb T)$,
		\begin{equation*}
		\big(e^{f(0)}- \chi\big) h(0) \, \psi(0) + \sum_{i=1}^{N_V-1} e^{f(a_i)} h(a_i) \, \psi(a_i) = 0,
		\end{equation*}
		which is clearly impossible.
		
		With respect to the Jacobian, let $A$  be a measurable  set such that $T\vert_A$ is injective. 
		For a sequence  $\{\psi_n\} \subset C^0(\mathbb T)$ converging to the indicator function on $ A $  $\nu$-almost every point, by the dominated convergence theorem, 
	\begin{align*}
	\int_A \chi  \frac{h\circ T}{h} e^{-f} \, d\mu &= \lim_{n\to \infty} \int \chi \,  h\circ T \, e^{-f} \, \psi_n \, d\nu
	                                                                    = \lim_{n\to \infty} \int \mathscr L_f( h\circ T\, e^{-f }\, \psi_n) \, d\nu \\ 
	                                                                 &= \lim_{n\to \infty} \int \mathscr L_f(e^{-f} \, \psi_n) \, d\mu= \mu(T(A)),
	\end{align*}
	since $ \mathscr L_f \big(e^{-f} \psi_n\big)(x)= \sum_{y\in T^{-1}(x)}\psi_n(y) \to \mathds 1_{T(A)}(x) $,  $\nu$-almost every $ x \in \mathbb T $.
	\end{proof}

It is well known that the topological pressure may be introduced by means of open coverings. 
We recall the main aspects of this formulation here and we refer the reader to \cite{Wal82} for more details. 
Given an open  cover  $\mathscr A$ of  $\mathbb T$, consider 
\begin{equation*}
p_n(T,f,\mathscr A):=\inf _{\mathscr B} \, \sum_{B \in \mathscr B}\exp\big(\sup_{x \in B} S_nf(x)\big),
\end{equation*}
where $\mathscr B$ is a finite subcover of $\mathbb T$  contained in $ \mathscr A \vee T^{-1} \mathscr A\vee \cdots \vee T^{-(n-1)} \mathscr A$.
Then the topological pressure may be defined as
\begin{equation*}
P(T,f):=  \lim_{\epsilon \to 0} \, \sup_{\text{diam}(\mathscr A) \le \epsilon} \, \lim_{n\to \infty}\frac{1}{n}\log p_n(T,f,\mathscr A).
\end{equation*}

\begin{lemma}\label{pressure log}
The following inequality holds: $ \log \chi \le P(T,f)$.
\end{lemma}

\begin{proof}
	Let $ \mathscr A $ be an open cover of $\mathbb T$ with diameter less than $ \varrho_0 $, the positive constant described in Lemma~\ref{desigualdade distancia}.
	If $\mathscr B$ is a finite subcover  of $\mathbb T$ contained in $  \vee_{j=0}^{n-1} T^{-j}(\mathscr A) $, by the very definition of  $ \varrho_0 $, for all $ x \in \mathbb T $ 
	any two distinct points of $ T^{-n}(x) $ belong to distinct elements of $\mathscr B$.
	Then 
	\begin{equation*}
	\chi^n=\chi^n \nu(\mathbb T)= \int \mathscr L_f^n \mathds 1 \, d\nu
	\le \int  \sum_{B \in \mathscr B} \exp(\sup_{B}S_nf) \, d\nu =  \sum_{B \in \mathscr B} \exp(\sup_{B}S_nf).
        \end{equation*}
        Taking the infimum among all finite subcovers  contained in $  \vee_{j=0}^{n-1} T^{-j}(\mathscr A) $, we obtain	
        $
	\log \chi \le \frac{1}{n} \log  p_n(T, f,\mathscr A),
	$
	which yields $\log \chi \le P(T,f) $.
\end{proof}

Given $ m \in M(\mathbb T, T) $ and a measurable function $ \phi : \mathbb T \to \mathbb R $, keeping in mind~\eqref{Property Jacobian}, consider now for $m$-almost every $ x \in \mathbb T $
\begin{equation*}
\mathscr J_{m}(\phi)(x):=\sum_{y\in T^{-1}(x)}\frac{1}{J_m(T)(y)} \phi(y).
\end{equation*}
We highlight two well-known main properties:
\begin{align}
		\int \phi \,  dm & = \int \mathscr J_{m}(\phi) \, dm,  \label{igualdade J} \\
		\int \mathscr J_{m}(\log \psi) \, dm & \le  \log \int \mathscr J_{m}(\psi) \, dm,  \label{desigualdade J}
\end{align}
for every measurable functions $ \phi, \psi : \mathbb T \to \mathbb R $ fulfilling integrability conditions. For details,
see \cite[Section 9.7]{VO16}.

For the next proposition, we also remark a basic fact: the eigenequation $ \mathscr L_f h = \chi h $ considered at the fixed point 
gives us $ \big( e^{f(0)} - \chi \big ) h(0) + \sum_i e^{f(a_i)} h(a_i) = 0 $, from which we conclude that
\begin{equation}\label{ponto fixo e autovalor}
f(0) < \log \chi.
\end{equation}

\begin{proposition} \label{Conditions 1 2 and 3}
The $T$-invariant probability  $\mu=h\nu$  is the unique equilibrium measure associated with $f$, and $$h_{\mu}(T)+\int f \, d\mu = \log \chi = P(T,f).$$
\end{proposition}

\begin{proof} 
        Let $ m $ be an equilibrium state associated with $ f $.
        We claim that $ m $ does not charge the indifferent fixed point.
        Indeed, replacing $m$ by one of its ergodic components if necessary, we can assume the $m$ is ergodic. 
	Suppose by contradiction that $m(\{0\})>0$. 
	Then for any measurable set $B$ with $m(B)=0$, we have $0\notin B$, so that the Dirac measure $\delta_0$ supported at the fixed point $0$ satisfies  $\delta_{0}(B)=0$, 
	which means that  $\delta_0$ is  absolutely continuous with respect to $m$.
	Birkhoff's ergodic theorem ensures that for a bounded measurable function $\phi$, 
	\begin{equation}\label{Birkhoff}
	\lim_{n\to \infty }\frac{1}{n} S_n(\phi)(x) =\tilde \phi(x), \qquad \text{  $m$-almost every $x\in\mathbb T$},
	\end{equation}
	where $\tilde \phi $ is $m$-almost everywhere constant and equals to $ \int \phi \, d m$. 
	Since $\delta_0$ is absolutely continuous with respect to $m$, equality~\eqref{Birkhoff} holds for $\delta_0$-almost every $x\in\mathbb T$. 
	In particular, $\int \phi \, d\delta_0 = \int \tilde\phi \, d\delta_0 = \int \phi  \, dm $, and we conclude that $m= \delta_0$. 
	Inequality~\eqref{ponto fixo e autovalor} and Lemma~\ref{pressure log} guarantee that  $\int f \, d\delta_0 < P(T,f)$. 
	Hence $m= \delta_0$ is not an equilibrium measure, which is  a contradiction. 
	
Hence by the $T$-invariance of  $m$, this probability  does not give any mass to the pre-images of $0$. 
Applying thus Lemma~\ref{Partition and Rokhlin formula},  $m$ admits a Jacobian $J_m(T)$ that satisfies  $h_m(T)= \int \log  J_m(T) \, dm$.
We use~\eqref{igualdade J} and~\eqref{desigualdade J} to see that
	\begin{align*}
	P(T,f) - \log \chi 
	= & \int \log  J_m(T) \, dm+ \int f\, dm - \log \chi \\
	= & \int \log \big(\chi^{-1} \frac{h}{h\circ T} \, e^f  J_m(T)\big) \, dm \\
	= &\int \mathscr J_{m}\Big(\log \big(\chi^{-1} \frac{h}{h\circ T} \, e^f  J_m(T)\big) \Big) \, dm \\
	\le & \log \int \mathscr J_{m} \big(\chi^{-1} \frac{h}{h\circ T} \, e^f  J_m(T)\big) \, dm \\
	= & \log \int \frac{1}{\chi h} \, \mathscr{L}_f(h) \,dm = 0.
	\end{align*}
	Therefore, together with Lemma~\ref{pressure log}, we get $P(T,f) = \log \chi.$

	 By Lemma~\ref{lem: lem 4}, the Jacobian of $T$ with respect to $\mu= h\nu$ is  $J_\mu(T)=
	\chi\, \frac{h\circ T}{h} e^{-f} $. Note then that  
	\begin{align*}
	h_\mu(T) = &\int \log J_{\mu}(T) \, d\mu = \int \log\big( \chi \, \frac{h}{h \circ T} \, e^{-f} \big) \, d\mu \\
	= & -\int f \, d\mu + \log \chi=  -\int f \, d\mu + P(T,f), 
	\end{align*}
	which shows that $ \mu $ is an equilibirum state.
	
	Concerning uniqueness, note first that when $ h_m(T) = \int \log J_m(T) \, dm $ (which we already showed to be necessarily the case for an equilibrium state), by~\eqref{igualdade J} we have
	$$ h_m(T) + \int f \, dm - P(T, f) = \int \mathscr J_m  \Big(  \log \big(\chi^{-1} \frac{h}{h\circ T} \, e^{f} \big) + \log J_{m}(T) \Big) \, dm. $$
	It is well known that if  $p_1,\cdots, p_n$ are nonnegative real numbers such that $\sum_i p_i =1$ and $b_1,\cdots, b_n$ are arbitrary real numbers, then 
	\begin{equation*}
	\sum_{i=1}^n \big( p_i b_i-p_i\log p_i \big) \le \log \big(\sum_{i=1}^n e^{b_i}\big)
	\end{equation*}	
	with equality only when $p_i= \frac{e^{b_i}}{\sum_j  e^{b_j}} $. 
	Therefore, thanks to~\eqref{Property Jacobian}, we get
	\begin{align*}
	\mathscr J_m  \Big(  \log \big(\chi^{-1} \frac{h}{h\circ T} \, e^{f} \big) + \log J_{m}(T) \Big) \le \log \big(\frac{1}{\chi h} \mathscr L_f h \big) = 0, 
	\end{align*}
	with  equality  if, and only if, $J_m(T)= \chi \, \frac{h\circ T}{h} \, e^{-f}= J_\mu(T)$, $m$-almost everywhere. 
	In particular, to obtain that $ \mu $ is the unique equilibrium state associated with $ f $, it suffices to argue that $ \mu $ is the only invariant probability that admits 
	$ \chi \, \frac{h\circ T}{h} \, e^{-f} $ as its (almost everywhere) Jacobian.	By definition  $\mathscr{J}^n_m(\phi)= \mathscr{J}^n_\mu(\phi)$ $m$-almost everywhere, for every continuous function $\phi$ and all $ n \ge 1 $.
        Besides, since $ \prod_{j=0}^{n-1}  J_\mu(T) \circ T^j = \chi^n \frac{h\circ T^n}{h} \, e^{-S_nf}$, we see that
        $$\mathscr J_{\mu}^n(\phi )=\frac{1}{\chi^{n} h}\mathscr L_f^n(\phi h). $$ 
	Thus, from Proposition~\ref{convergence of iterates},  $ \mathscr J_{\mu}^n(\phi )  \to  \int \phi \, d\mu$ uniformly as $ n $ tends to infinity.  
	However, by~\eqref{igualdade J} $ \int \phi \, dm = \int \mathscr J_m^n (\phi) \, dm$.
	Hence, applying the dominated convergence theorem, we have,  for any continuous function $\phi $,  $\int \phi \, dm= \int \phi \, d\mu $, so that $m= \mu$.
\end{proof}
	
	\subsubsection*{Gibbs measure}
	
	In order to show the Gibbs' property obeyed by the probability $ \mu = h \nu $, 
	for every $x \in \mathbb T$,  $r>0$ and $n \ge 0$, define the corresponding dynamic ball 
	$$ B(x, n, r):= \{y \in \mathbb T: d(T^j(x), T^j(y)) < r, \, \, j=0,1,\cdots,  n\}. $$
	
	\begin{proposition}\label{Prop1} 
	The equilibrium state $\mu$ is a Gibbs measure: given $ r\in(0,\varrho_1) $ (where $\varrho_1$ is the constant  from Definition~\ref{modulosepreimagens}), 
	there exists a constant $ K_r > 0 $ such that  for  $ x\in \mathbb T$ and $ n \ge 1$
		$$K_r^{-1} \le \frac{\mu\big(B(x, n,r)\big)}{e^{S_n f(x)-nP(T, f)}}\le K_r.$$
	\end{proposition}
	
	\begin{proof}
		It is well known that the Jacobian of $ T^n $, $ n \ge 1 $, with respect to $ \mu $ may be described (almost everywhere) as $  \prod_{j=0}^{n-1}  J_\mu(T) \circ T^j  $.
		Therefore, thanks to Lemma~\ref{lem: lem 4} and  Proposition~\ref{Conditions 1 2 and 3}, we have 
		\begin{equation}\label{eq: n jacobian}
		J_\mu(T^n)= e^{n P(T, f)} \frac{h\circ T^n}{h} e^{-S_n f}.
		\end{equation}
		Let $r\in (0, \varrho_1)$,  $x\in \mathbb T$ and $ n \ge 1$.
		Then for $z \in B(x, n, r)$, since $ d(T^n(z), T^n(x)) < \varrho_1 $, there is a unique pre-orbit  $\{ z_n\}$ of $T^n(z)$ satisfying the properties in Definition~\ref{modulosepreimagens}.
		In particular,   $ z_{n-k} = T^k(z)$ for $ k = 0, 1, \ldots, n $. Hence, from~\eqref{variation sum}, we conclude that
		\begin{equation}\label{eq: sum variation}
		\vert S_n f(x)-S_nf(z) \vert \le \kappa_f \, \Omega(1/2), \qquad \forall \, z \in  B(x, n, r).
		\end{equation}
	        By Lemma~\ref{desigualdade distancia},  $T^n\vert_{B(x,n,r)}$ is injective. 
		Moreover, being an element of $ \Lambda $, it follows from inequality~\eqref{extensao propriedade local de Lambda} that 
		the eigenfunction $ h $ satisfies $e^{-L \kappa_f \Omega(1/2)} \le  \frac{h(T^n(z))}{ h(z)}\le e^{L \kappa_f \Omega(1/2)}$ for any $ z \in \mathbb T $ and $ n \ge 1 $. 
		Hence, using~\eqref{eq: n jacobian} and~\eqref{eq: sum variation}, we have
		\begin{align}
		 \mu\big(T^n  (B(x,n,r)) \big)
		&=\int_{B(x,n,r)}  J_\mu(T^n)(z) \, d\mu(z) \nonumber \\
		&=\int_{B(x,n,r)} e^{n P(T, f)-S_n f(z)} \frac{h\circ T^n(z)}{h(z)} \, d\mu(z). \nonumber \\
		& \ge e^{-L \kappa_f \Omega(1/2)} e^{nP(T, f)- S_n f(x)}\int_{B(x,n,r)} e^{S_n f(x)-S_nf(z)} \,  d\mu(z) \nonumber \\
		&\ge K^{-1} \,\frac{\mu(B(x,n,r))}{e^{S_n f(x)-nP(T, f)}}, \label{majoracaoGibbs}
		\end{align}
		where $K = e^{(L+1)\kappa_f \Omega(1/2)}$. Similarly
		\begin{align}
		 \mu\big(T^n  (B(x,n,r)) \big)
		& \le e^{L \kappa_f \Omega(1/2)} e^{nP(T, f)- S_n f(x)}\int_{B(x,n,r)} e^{S_n f(x)-S_nf(z)} \,  d\mu(z) \nonumber \\
		&\le K \,\frac{\mu(B(x,n,r))}{e^{S_n f(x)-nP(T, f)}}. \label{minoracaoGibbs}
		\end{align}
		
		Since the uniqueness of pre-orbits in Definition~\ref{modulosepreimagens} ensures that $ T^n (B(x,n,r)) = T (B(T^{n-1} (x),1,r)) $, a particular application of~\eqref{majoracaoGibbs} and~\eqref{minoracaoGibbs} shows that the value 
		$  \mu\big(T^n  (B(x,n,r)) \big) $ belongs to the interval
		$$  \Big( \frac{K^{-1}}{e^{\max f - P(T,f)}} \, \inf_{y \in \mathbb T} \mu(B(y, 1, r)), \,  \frac{K}{e^{\min f - P(T,f)}} \, \sup_{y \in \mathbb T} \mu(B(y, 1, r))  \Big). $$
		Hence to complete the proof, it remains to argue that $  \inf_{y \in \mathbb T} \mu(B(y, 1, r)) > 0 $. 
		In fact, as the dynamics is topologically exact, for each $ y \in \mathbb T $, 
		there exists a positive integer $ M_y $ such that the restriction of $ T^{M_y -1} $ on $ B(y, 1, r) $ is injective and has image strictly contained in $  \mathbb T $, and also that $ T^{M_y} (B(y, 1, r)) = \mathbb T $.
		By continuity, $ M_y $ is locally constant: for any $ \hat y $ sufficiently close to $ y $, we have $ M_{\hat y} = M_y $.
		By compactness, one may find a finite cover $ \{ B(y_i, 1, r) \} $ of $ \mathbb T $ with $ M_i := M_{y_i} $ constant on each $ B(y_i, 1, r) $. 
		Now, for an arbitrary $ y \in \mathbb T $, consider $ i $ such that $ y \in B(y_i, 1, r) $ as well as the corresponding $ M_i $.
		Fix then a half-open arc $ A_y \subset   B(y, 1, r) $ for which $ T^{M_i} : A_y \to \mathbb T $ is bijective.
		Once again taking advantage of Jacobians, we see that
		\begin{align*}
		1 = \mu(\mathbb T) & = \mu\big( T^{M_i}(A_y )\big) \\
		& \le e^{L \kappa_f \Omega(1/2)} e^{M_i P(T, f)} \int_{A_y} e^{-S_{M_i} f(z)} \,  d\mu(z) \\
		& \le e^{L \kappa_f \Omega(1/2)} e^{M_i (P(T, f) - \min f)} \, \mu(A_y),
		\end{align*} 
		which yields
		$$ \inf_{y \in \mathbb T} \mu(B(y, 1, r)) \ge  e^{- L \kappa_f \Omega(1/2)} \min_i e^{M_i (\min f - P(T, f))} > 0. $$
	\end{proof}

\footnotesize{

}

\end{document}